\documentclass[
final
 , nomarks
]{dmtcs-episciences}

\usepackage[utf8]{inputenc}
\usepackage{subfigure}

\newtheorem{theorem}{Theorem}[section]

\newtheorem{lemma}[theorem]{Lemma}

\makeatletter
\newcommand\figcaption{\def\@captype{figure}\caption}
\newcommand\tabcaption{\def\@captype{table}\caption}
\makeatother

\newtheorem{conjecture}[theorem]{Conjecture}

\newcommand{\remark}{\medskip\par\noindent {\bf Remark.~~}}

\newcommand{\mad}{{\rm mad}}
\newcommand{\moddd}{{\rm mod} \ }

\usepackage[round]{natbib}

\author[Oothan Nweit \and Daqing Yang]{Oothan Nweit \thanks{E-mail: \small\texttt{oothannweit@gmail.com}.} 
  \and Daqing Yang \thanks{Corresponding author,  grant numbers: 
  	NSFC 12271489, U20A2068. E-mail: \small\texttt{dyang@zjnu.edu.cn}.}}
\title[On the mod $k$ chromatic index of graphs]{On the mod $k$ chromatic index of graphs} 
\affiliation{
  School of Mathematical Sciences, Zhejiang Normal University, Jinhua, Zhejiang, China} 
\begin{document}

% This is only used if you are compiling for a volume before vol 25
% \publicationdetails{VOL}{2015}{ISS}{NUM}{SUBM}
% This is the new form of collecting the data, starting with vol 25
%\publicationdata
\publicationdata
{vol. 26:3}{2024}{16}{10.46298/dmtcs.13187}{2024-03-07; 2024-03-07; 2024-10-22}{2024-10-24}
%{vol. 25:3 special issue for main purpose}
%{2022}
%{1}
%{10.46298/dmtcs.10472}
%{1998-10-14; 1998-10-14; 2002-07-19; 2014-02-05; 2015-09-09; 2022-12-25}
%{2022-12-3}
%{2022-12-3; None}
%{2023-1-1}
\maketitle
\begin{abstract}
  For a graph $G$ and an integer $k\geq 2$, a $\chi'_{k}$-coloring of $G$ is an edge coloring of $G$ such that the subgraph induced by the edges of each color has all degrees congruent to $1 ~ (\moddd k)$, 
  and $\chi'_{k}(G)$ is the minimum number of colors in a $\chi'_{k}$-coloring of $G$.  
  In [\lq\lq The mod $k$ chromatic index of graphs is $O(k)$", J. Graph Theory. 2023; 102: 197-200], Botler, Colucci and Kohayakawa proved that $\chi'_{k}(G)\leq 198k-101$ for every graph $G$. In this paper, we show that $\chi'_{k}(G) \leq 177k-93$.  
\end{abstract}

%%%%%%%%%%%%%%%%%%%%%%%%%%%%%%%%%%%%%%%%%%%%%%%%%%%%%%%%%%%%%%%%%%%%%%%%%%

\section{Introduction}
All graphs considered here are simple.  
Let $G=(V, E)$ be a graph,  and $v(G):=|V(G)|$ and $e(G):=|E(G)|$.  
If $X \subseteq V (G)$, then $G[X]$ is the subgraph of $G$ induced by $X$.
% and $G-X$ is the graph obtained from $G$ by deleting vertices in $X$.  
% For a graph $G=(V, E)$ and 
For an integer $k \geq 2$, a $\chi'_{k}$-coloring of $G$ is a coloring of the edges of $G$ such that the subgraph induced by the edges of each color has all degrees congruent to $1 ~ (\moddd k)$, 
and the $\moddd k$ chromatic index of graph $G$, denoted by $\chi'_{k}(G)$, is the minimum number of colors in a $\chi'_{k}$-coloring of $G$. \cite{pyber} proved that $\chi'_{2}(G)\leq 4$ for every graph $G$ and asked whether $\chi'_{k}(G)$ is bounded by some function of $k$ only. \cite{scott} proved that $\chi'_{k}(G)\leq 5k^{2}\log k$ for any graph $G$, and in turn asked if  $\chi'_{k}(G)$ is in fact bounded by a linear function of $k$.   
%Botler, Colucci, and Kohayakawa 
\cite{botler} answers Scott's question affirmatively by proving  the following theorem. 
% that $\chi_{k}'(G) \leq 198k-101$, this the authors 
\begin{theorem} [\cite{botler}] \label{BCK} 
	For every graph $G$ we have $\chi'_{k}(G) \leq 198k-101$. 
\end{theorem} 

Also in \cite{botler}, Botler, Colucci, and Kohayakawa proposed the following conjecture: 
\begin{conjecture} [\cite{botler}] \label{conj}
	There is a constant $C$ s.t. $\chi'_{k}(G) \leq  k + C$ for every graph $G$.  	
\end{conjecture} 

%$\forall$ graph $G$.
%every   
%Note that the 
%The multiplicative constant in Conjecture \ref{conj} would be best possible, as  $\chi'_{k}(K_{1, k})  = k$. % where $K_{1, k}$.

In this paper, we improve the upper bound of the mod $k$ chromatic index of graphs by proving the following theorem. 
\begin{theorem}  \label{Main-thm} 
	For every graph $G$ we have $\chi'_{k}(G) \leq 177k-93$.  
\end{theorem}

%For a graph $G$, we let $e (G)$ denote the number of edges and $v (G)$ denote the number of vertices, of $G$.  

%For a graph $G=(V, E)$,  let $v(G):=|V(G)|$ and $e(G):=|E(G)|$. 
In the proof of Theorem \ref{BCK},  %the authors 
%Botler et al. in  
\cite{botler} applies the  following two lemmas. 

\begin{lemma}[\cite{Mader}] \label{Mader}
	If $k \ge 1$, $G$ is a graph with $e (G) \ge 2kv(G)$, then $G$
	contains a $k$-connected subgraph.  
\end{lemma}

\begin{lemma}[\cite{Thomassen}] \label{Thomassen}
	If $k \ge 1$ and G is a $(12k - 7)$-edge-connected graph with an 
	even number of vertices, then $G$ has a spanning subgraph in which each vertex has degree congruent to $k ~ (\moddd 2k)$. 
\end{lemma}

A graph $G$ is {\em $k$-divisible} if $k$ divides the degree of each vertex of the graph $G$. 
By applying Lemma \ref{Mader} and Lemma \ref{Thomassen}, 
%Botler et al  
\cite{botler} proved the following lemma.

\begin{lemma}[\cite{botler}] \label{botler1} 
	If graph $G$ does not contain a nonempty $k$-divisible subgraph, then $e(G) < 2(12k - 6)v(G)$. 
\end{lemma}

For a graph $G$, let $N_G(v)$ denote the neighbors of $v$, $E_G(v)$ denote the edges  that are incident to $v$, and let $d_G(v)$ be the degree of $v$, i.e., $d_G(v)=|E_G(v)|$.  
Let $\vec{G} = (V, \vec{E})$ be  an orientation of $G$,  
for $v \in V$,
let $N_{\vec{G}}^+(x)$ denote the out-neighbor(s) of $x$, i.e.,
$N_{\vec{G}}^+(x) = \{y: x \rightarrow y \}$, let $d_{\vec{G}}^+(x)$ be the
out-degree of $x$, i.e., $d_{\vec{G}}^+(x) = |N_{\vec{G}}^+(x)|$; 
if $y$ is an out-neighbor of $x$, then we say edge $\overrightarrow{xy}$ an out-edge of $x$. 
Let $N_{\vec{G}}^-(x)$ denote the in-neighbor(s) of $x$, i.e.,
$N_{\vec{G}}^-(x) = \{y: x \leftarrow y \}$, let $d_{\vec{G}}^-(x)$ be the
in-degree of $x$, i.e., $d_{\vec{G}}^-(x) = |N_{\vec{G}}^-(x)|$; 
if $y$ is an in-neighbor of $x$, then we say edge $\overleftarrow{xy}$ an in-edge of $x$. 
Let $\Delta^{+}\left( \vec{G}\right) = \max_{v \in V}
d^+_{\vec{G}}(v)$, $\Delta^{-}\left( \vec{G}\right) = \max_{v \in V}
d^-_{\vec{G}}(v)$. 
We drop the subscripts $G$ or $\vec{G}$ in the above notations when $G$ or $\vec{G}$
is clear from the context. 

The {\em maximum average degree} of a graph $G$, denoted by $\mad(G)$, is defined as  
$$\mad(G) = \max_{H \subseteq G} \frac {2e(H)} {v(H)},$$
which places a bound on the average vertex degree in all subgraphs. 
It has already attracted a lot of attention and has a lot of applications.   
The following theorem is well-known (cf. \cite{SL-hak}, Theorem 4), we use it in our proof of Theorem \ref{Main-thm}.   

\begin{theorem} \label{decom}
	Let $G$ be a graph. Then  $G$ has an orientation $\vec{G}$ 
	such that $\Delta^{+}(\vec{G} ) \le d$ if and only if 
	$\mad (G) \leq 2d$. 
\end{theorem} 

\section{Proof of Theorem \ref{Main-thm}}   

In this section, we prove Theorem \ref{Main-thm}.
%  this improves the upper bound of $\chi'_{k}(G)$ %for graph $G$ 
% given in Theorem \ref{BCK}. already been
The proof uses the following lemma, which was  available in \cite{Yang}.  
For the completeness of this paper, we present its short proof here. 

%If $G$ is a graph with $\mad(G) \leq 2d$,   
%If  graph  $G$ has an orientation $\vec{G}$ such that $\Delta^{+}(\vec{G} ) \le d$,  

%for $v \in V$, let $\partial_{\vec{G}}^{-}(v)$ denote   the  set of directed edges in $\vec{E}$ with $v$ as their head, $\partial_{\vec{G}}^{+}(v)$ denote the set of directed edges in $\vec{E}$ 
%% arcs in $D$ 
%with $v$ as their tail; we name $\partial_{\vec{G}}^{-}(v)$ the in-edges of $v$, and  $\partial_{\vec{G}}^{+}(v)$ the out-edges of $v$ (in $\vec{G}$).    
%Let $d_{\vec{G}}^{-}(v) = |\partial_{\vec{G}}^{-}(v)|$ and $d_{\vec{G}}^{+}(v) = |\partial_{\vec{G}}^{+}(v)|$,      

\begin{lemma} (\cite[Lemma 1.7]{Yang}, adapted) \label{HM-LL} 
	Let $d \ge 0$ be an integer. 
	If an oriented graph  $\vec{G}$ has  $\Delta^{+}(\vec{G} ) \le d$,   
	then there exists a linear order $\sigma$ of $V(\vec{G})$, such that for any  vertex 	$u \in V(\vec{G})$, the number of vertices that are the in-neighbors of  $u$, and precede $u$ in $\sigma$ is at most $d$.   
\end{lemma}    
\begin{proof} 	
	We recursively construct a linear ordering $\sigma = v_{1}v_{2} \ldots v_{n}$ of
	$V = V(\vec{G})$ as follows. Suppose that we have constructed the
	final sequence $v_{i+1} \ldots v_{n}$ of $L$. (If $i=n$ then this sequence is empty.) Let $M=\left\{ v_{i+1}, \ldots, v_{n} \right\} $ be the set of vertices
	that have already been ordered and $U=V-M$ be the set of vertices
	that have not yet been ordered. Let $\vec{G_{U}} \subseteq \vec{G}$ 
	be the subgraph of $\vec{G}$ induced by $U$. If we have not yet finished 
	constructing $\sigma$, we choose $v_{i}\in U$ so that 
	$d_{\vec{G_U}}^-(v_i)$ is minimal in $\vec{G_U}$. %Note that  
	Since  $\sum_{v \in U} d_{\vec{G_U}}^{-}(v) = \sum_{v \in U}  d_{\vec{G_U}}^{+}(v)$,   and $\Delta^{+}\left(\vec{G_U}\right) \leq \Delta^{+}\left( \vec {G}\right) \le d$,  we have  $d_{\vec{G_U}}^-(v_i) \leq  d$. This proves the lemma.    
\end{proof}

Combining Lemma \ref{HM-LL} and  the techniques used in \cite{botler}, we prove the following lemma, which is the key to the proof of Theorem \ref{Main-thm}.   

%The proof indeed combines the Harmonious Strategy and the techniques used in \cite{botler}, which provides a new and interesting application of the Harmonious Strategy to a non-game problem.    
%
%\begin{lem} \label{botler3.1}
% Let $d \ge 0$ be an integer and $G$ be a graph with $\Delta^{*}(G)\leq d$, then $\chi'_{k}(G)\leq 7d+2k-2$.
%\end{lem} 

\begin{lemma} \label{Key}
	Let $d \ge 1$ be an integer. If a graph $G$ has an orientation $\vec{G}$ such that  $\Delta^{+}(\vec{G} ) \le d$, then 
	$\chi'_{k}(G)\leq 7d+2k-3$. 
\end{lemma}

%  is a graph with $\mad(G) \leq 2d$, then  
% Let $V(G)=\{v_{1}, v_{2},\ldots, v_{n}\}$ be an ordering of $V(G)$, which is the order in which the vertices of $G$ are processed during the Harmonious Strategy.

% If $G$ is a graph with $\mad(G) \leq 2d$, then by Theorem \ref{decom},    $G$ has an orientation $\vec{G}$ such that $\Delta^{+}(\vec{G} ) \le d$.  		

\begin{proof} 	
	If a graph $G$ has an orientation $\vec{G}$ such that  $\Delta^{+}(\vec{G} ) \le d$, by applying Lemma \ref{HM-LL}, we can suppose   
	linear ordering $\sigma := v_{1}v_{2} \ldots v_{n}$ of $V(G)$ satisfying that for any  vertex 	$u \in V(\vec{G})$, the number of vertices that are the in-neighbors of  $u$, and precede $u$ in $\sigma$ is at most $d$.  
	
	%For a directed edge $\overrightarrow{vu}$ of  $\vec{G}$, if $v$ precede $u$ in $\sigma$, we say 
	
	%By using 
	%we color $G$ by coloring the edges incident with $v_{i}$ for each $v_i \in \{v_1, \ldots, v_{n-1}\}$ in turn, 
	
	Following the above linear ordering $\sigma$,  we give a $\chi'_{k}$-coloring   of  $G$ by coloring the edges incident with $v_{i}$ for each $v_i \in \{v_1, \ldots, v_{n-1}\}$ in turn.   
	At step $v_i$, we name  {\em this procedure as processing vertex $v_i$, which means  that we color all the edges incident with $v_{i}$ that are not colored yet at this time}. 
	After we have finished processing vertex $v_i$, we shall maintain that we have a $\chi'_{k}$-coloring of the graph spanned by the edges incident with $v_{1}, \ldots, v_{i}$, we call this a  {\em good partial $\chi'_{k}$-coloring after step $v_i$}. 
	
	%such that  %at each step $i$,  
	%after we have finished step $i$, 
	
	% when $v_{j}$ is processed.  
	%

	%The coloring method is as follows:  arbitrarily
	To define the  coloring method, 
	we partition all the colors  into two sets $C_{1}$ and $C_{2}$  such that $|C_{1}|=3d-1$ and $|C_{2}| = 4d+2k-2$, note that $|C_{1}| + |C_{2}| = 7d+2k-3$.  For  each $1\leq i \leq n-1$, we use the colors in $C_{1}$ to color the uncolored out-edges of $v_{i}$ and use the colors in $C_{2}$ to color the uncolored in-edges of $v_{i}$. 
	Equivalently, for any directed edge $uv$ $(u \rightarrow v$ in $\vec{G})$, if $u$ is processed before $v$, then edge $uv$ is colored with a color in $C_{1}$;   if $v$ is processed before $u$, then edge $uv$ is colored with a color in $C_{2}$.
	%  it is colored with a color in $C_{2}$. 

	By induction on $i$, we give a good partial $\chi'_{k}$-coloring of  $G$ after    processing $v_i$. % in turn.    
	For the induction hypothesis, suppose when we begin to process vertex $v_i$, all the edges  that are incident with a vertex $v$ that precedes $v_i$ in $\sigma$ have already been colored.
	% and used colored in $C_1$ or $C_2$  as described above.  
	
	For each $1 \leq i \leq n-1$, when $v_{i}$ is processed, let $U^{+}(v_i)$ denote the unprocessed out-neighbor(s) of $v_{i}$, i.e.,  $U^{+}(v_i) = N_{\vec{G}}^+(v_i)  \cap  \{v_{i+1}, \ldots, v_{n}\}$;  let $U^{-}(v_i)$ denote the unprocessed in-neighbor(s) of $v_{i}$, i.e.,  $U^{-}(v_i) = N_{\vec{G}}^-(v_i)  \cap \{v_{i+1}, \ldots, v_{n}\}$.  
	When we process  $v_{i}$, we color the uncolored out-edges $\{v_{i}v_{j}: v_{j} \in  U^{+}(v_i)\}$ such that all out-edges of $v_i$ are colored with distinct colors;  
	% in $C_1$  or $C_2$, 
	to  color the uncolored in-edges $\{v_{j}v_{i} : v_{j}\in U^{-}(v_i)\}$, we use colors in $C_2$ that are different than having been used in the out-edges of  $v_i$ (refer $X(v_i)$ in the following paragraph);  and we do the coloring in this order.  
	
	For the induction step, suppose we process vertex $v_i$, where  $i \in \{1, 2, \ldots, n-1\}$. 
	Suppose $N_{\vec{G}}^{+}(v_{i}) \cap \{v_{1}, \ldots, v_{i-1}\}=\{x_{1}, \ldots, x_{\ell}\}=X(v_i)$. Since $\Delta^{+}(\vec{G} ) \le d$,	%we have 
	$|X(v_i)| \leq  d$.     
	Suppose $N_{\vec{G}}^{-}(v_{i})\cap \{v_{1}, \ldots, v_{i-1}\}=\{y_{1}, \ldots, y_{r}\} = Y(v_i)$, then by Lemma \ref{HM-LL}, $|Y(v_i)| \leq d$.    
	For the induction  hypothesis, we suppose each edge $y_{j}v_{i}$, where $y_{j} \in Y(v_i)$, is  colored with a color in $C_{1}$; and each edge $v_{i}x_{j}$, where $x_{j} \in X(v_i)$, is colored with a color in $C_{2}$.  
	Now we process vertex $v_i$, i.e., color the remaining uncolored edges incident with $v_{i}$, we do this in two steps.  
	
	In the first step, we use the colors in $C_{1}$ to color the uncolored out-edges $\{v_{i}v_{j}: v_{j} \in U^{+}(v_i)\}$, such that all the out-edges of $v_i$ have distinct colors. 
	We show that we can do this for any edge $v_{i}v_{j}$ with $v_{j} \in U^{+}(v_i)$.   
	For vertex $v_j$, since  $v_{j} \in U^{+}(v_i)$, $v_j$ has not been processed yet. 
	Therefore, by  Lemma \ref{HM-LL} and the induction hypothesis, the in-edges of $v_j$ that have been colored with colors in $C_{1}$ is at most $d-1$ (note that in the counting we removed the in-edge $v_iv_j$ of $v_j$).   
	For vertex $v_i$, when we begin to process vertex $v_i$, by the induction hypothesis, the edges incident with $v_i$ and colored with colors in $C_1$ are  $y_{j}v_{i}$, where $y_{j} \in Y(v_i)$. During  processing vertex $v_i$, for the edge $v_{i}v_{j}$ with $v_{j} \in U^{+}(v_i)$, at most $|U^{+}(v_i)|-1$  edges incident with $v_i$ are colored with colors in $C_1$. 
	Note that,   
	$$|C_{1}|-|Y(v_i)|-(|U^{+}(v_i)|-1)-(d-1) \geq (3d-1) - d - (d-1) - (d-1) \ge 1.$$ 
	This proves that there is a color left in $C_1$ for $v_{i}v_{j}$ with $v_{j} \in  U^{+}(v_i)$.   
	
	% since at this time $v_{i}^{+}$ is an unprocessed vertex, and by Lemma \ref{HM-LL} and the induction hypothesis, at most  $d_{\vec{G}}^{+}(v_{i}^{+})\leq d$ of the in-edges of $v_{i}^{+}$ are colored with  the color in $C_{1}$, so we have $|C_{1}|-|Y(v_i)|-(|U^{+}(v_i)|-1)-d \geq 3d-d-(d-1)-d>0$.

	%We can do this, since at this time $v_{i}^{+}$ is an unprocessed vertex, and by Lemma \ref{HM-LL} and the induction hypothesis, at most  $d_{\vec{G}}^{+}(v_{i}^{+})\leq d$ of the in-edges of $v_{i}^{+}$ are colored with  the color in $C_{1}$, so we have $|C_{1}|-|Y(v_i)|-(|U^{+}(v_i)|-1)-d \geq 3d-d-(d-1)-d>0$. 

	%the edges between these vertices and $v_{i}$ form the out-edges of $v_{i}$ and colored with the colors in $C_{2}$, 
	
	In the second step, we color the uncolored in-edges $R(v_i) =  \{v_{j}v_{i} : v_{j} \in U^{-}(v_i)\}$ of $v_{i}$. 
	Note that $v_{i}$ has at most $|X(v_i)| \leq d$ processed out-neighbors before $v_{i}$ is processed.  
	Observe that for any edge $x_ix_{\ell'}$ with $x_{\ell'} \in X(v_i)$,   $x_ix_{\ell'}$ is an in-edge of $x_{\ell'}$. By the induction hypothesis, $x_ix_{\ell'}$ is colored when $x_{\ell'}$ is processed, and is colored with a   color in $C_{2}$.  
	After removing the colors that used by edges $x_ix_{\ell'}$ with $x_{\ell'} \in X(v_i)$, there are at least $|C_2| - |X(v_i)| \ge 4d+2k-2 -d = 3d+2k-2$ colors left in  $C_{2}$ that can be used to color edges in $R(v_i)$.

	We partition these left colors in $C_2$ arbitrarily into sets $A(v_i)$ and $B(v_i)$ so that $|A(v_i)| = d+k$ and $|B(v_i)| \geq 2d+k-2$. For each $v_{j} \in U^{-}(v_i)$  is an unprocessed in-neighbor of $v_{i}$,
	we say that a color $c$ is $forbidden$ at $v_{j}$ if there is out-edge of $v_{j}$ is colored with $c$, and we call the colors in $A(v_i)$ that are not forbidden at $v_{j}$ $available$ at $v_{j}$. 
	Note that at most $d-1$ out-edges of $v_{j}$ are colored (removing the out-edge $v_jv_i$ of $v_j$ in the counting). This implies that at least $k+1$ colors in  $A(v_i)$ are available at $v_{j}$.

	Let $R^{*}(v_i)$ be the maximal subset of $R(v_i)$ that can be colored with colors in $A(v_i)$ in a way such that:   
	\begin{itemize}
		\item[(a)]  each in-edge $v_{j}v_{i}\in R^{*}(v_i)$ of $v_i$ is colored with a color available at $v_{j}$;
		\item[(b)] the number of edges in $R^{*}(v_i)$ colored with any color is   congruent to $1~(\moddd k)$. 
	\end{itemize} 
	Let $\bar{R}(v_i) = R(v_i) \setminus R^{*}(v_i)$ be the set of the remaining edges in $R(v_i)$. We  claim  that $|\bar{R}(v_i)| < |A(v_i)|$.    
	
	Assume otherwise that $|\bar{R}(v_i)| \geq |A(v_i)|$,  and suppose $A(v_i) =\{a_{i} : 1 \le i \le d+k\}$, $\bar{R}(v_i) = \{e_{j}=w_{j}v_{i} : w_{j} \in U^{+}(v_i),  ~ 1 \le j \le t, \mbox{ and } t\geq d+k\}$.   
	%Now we can construct the 
	We define an auxiliary bipartite graph  $T$
	%$T[A(v_i), \bar{R}(v_i)]$ 
	with vertices bipartition $A(v_i)$ and  $\bar{R}(v_i)$, edges   
	$E(T) = \{a_{i}e_{j} : \mbox{where $e_{j}=w_{j}v_{i}$, $a_{i}$ is available at } w_{j}\}$.  
	
	Since, for each  $w_{j} \in U^{+}(v_i)$,  there are  at least $k+1$ colors in   $A(v_i)$ available at $w_{j}$,   we have  $d_{T}(e_{j}) \geq k+1$ for every $e_{j} \in \bar{R}(v_i)$. 
	Therefore, 
	\[
	\sum_{a_i \in A(v_i)} d_{T}(a_{i})   = |E(T)| = \sum_{e_j \in  \bar{R}(v_i)}d_{T}(e_{j}) \geq (k+1)t \ge (k+1)(d+k). 
	\]
	Since $|A(v_i)| = d+k$, we concluded that there exists a color $a_{i}$ in $A(v_i)$, $d_{T}(a_{i}) \ge k+1$, which means that color $a_{i}$ is available on at least $k + 1$ edges in $\bar{R}(v_i)$. 
	
	If some edge in $R^{*}(v_i)$ is already colored with $a_{i}$, then we color $k$  edges in $\bar{R}(v_i)$ with color $a_{i}$. If no edge in $R^{*}(v_i)$ is colored  with $a_{i}$, then we color $k+1$ edges in $\bar{R}(v_i)$ with color $a_{i}$. Both of these cases contradict with the maximality of $R^{*}(v_i)$. This proves that  $|\bar{R}(v_i)| < |A(v_i)| = d+k$. 
	
	Finally we show that we can color all the edges in $\bar{R}(v_i)$ with distinct colors in $B(v_i)$. 
	For this, it suffices to note that, for each $w_{j}v_{i} \in \bar{R}(v_i)$, there  are at most $d-1+|\bar{R}(v_i)|-1 \leq 2d+k-3<|B(v_i)|$ colors of $B(v_i)$ that are either forbidden at $w_{j}$, or were used on previous edges of $\bar{R}(v_i)$.
\end{proof}

%The main result of this paper can be obtained from the above lemma as follows. 

%Our main result can be obtained from 

We prove our main result by using Lemma \ref{botler1}, Theorem \ref{decom}, Lemma \ref{HM-LL}, and Lemma \ref{Key}. The proof is similar to Theorem 5 in \cite{botler}, 
the differences are applications of Theorem \ref{decom} and Lemma \ref{HM-LL} here,  and  Lemma \ref{Key} is stronger than the corresponding one in \cite{botler}.  

% 
%\begin{theorem} \label{every}
%For every graph $G$ we have $\chi'_{k}(G) \leq 170k - 87$.  
%\end{theorem}

\bigskip
\noindent
{\bf Theorem \ref{Main-thm}}.
{\it For every graph $G$ we have $\chi'_{k}(G) \leq 177k - 93$.}
\bigskip

\begin{proof}   
	Let $H$ be a maximal subgraph of $G$ such that $d_{H}(v) \equiv 1 ~ (\moddd k)$ for every $v\in V(H),$ and let $G'= G\backslash E(H)$. Then  $V(G) \setminus V(H)$ is independent. Since otherwise, %assume that 
	there exists an edge $e$ with both ends in $V(G)\setminus V(H)$; then $H' = H+e$ would be a graph for  which $d_{H'}(v) \equiv 1 ~ (\moddd k)$; but this contradicts the maximality of $H$.  
	
	Similarly, by the maximality of $H$, $G'[V(H)]$ has  no nonempty $k$-divisible  subgraph.  	 
	By Lemma \ref{botler1}, for every nonempty $J \subseteq  G'[V(H)]$, we have $e(J) <  2(12k-6) v(J)$. Thus,  
	\[
	\mad(G'[V(H)])  = \max_{J \subseteq G' [V(H)]} 
	\frac{2 e(J)}{v(J)}  
	< \max_{J\subseteq G' [V(H)]}\frac{2 (24k - 12)v(J)}{v(J)}    
	= 2(24k - 12).
	\]		
	By Theorem \ref{decom},  $ G'[V(H)]$ has an orientation   $\overrightarrow{G'[V(H)]}$ such that   $\Delta^{+}(\overrightarrow{G'[V(H)]}) \le 24k-12$.  
	
	For every vertex $u \in V(H)$, by the maximality of $H$, $u$ has at most $k-1$ neighbors in $ V(G) \setminus V(H)$. 	
	For every edge $e = uv $ in $G'$ with $u \in V(H)$ and $v \in V(G) \setminus V(H)$, we orient $e$ from $u$ to $v$.   
	
	Thus there exists an orientation  $\vec{G'}$ of $G'$, such that $\Delta^{+}(\vec{G'}) \le  25k-13$.  By Lemma  \ref{Key},  
	there exists a $\chi'_{k}$-coloring of $G'$ using at most $177k-94$ colors. Then color all $E(H)$ with a new color, this proves the theorem.  
\end{proof} 

%Note that  
%all the vertices in  $V(G) \setminus V(H)$ has out degree $0$; $ G'[V(H)]$ has an orientation   $\overrightarrow{G'[V(H)]}$ such that   $\Delta^{+}(\overrightarrow{G'[V(H)]}) \le 24k-12$; for every vertex $u \in V(H)$, $u$ has at most $k-1$ neighbors in $ V(G) \setminus V(H)$.  

\remark In the above proof of Theorem \ref{Main-thm},  for all the edges $e = uv $ in $G'$ with $u \in V(H)$ and $v \in V(G) \setminus V(H)$, we can orient $e$ from $u$ to $v$.      
Then define a linear ordering $\sigma'$ beginning with vertices in $ V(G) \setminus V(H)$, and concatenating a linear ordering of vertices in $G'[V(H)]$ that has been  proved existing by Lemma \ref{HM-LL}, but use $\Delta^{+}(\overrightarrow{G'[V(H)]})  \le 24k-12$ here (instead of using of $\Delta^{+}(\vec{G'}) \le  25k-13$ as the proof of Theorem \ref{Main-thm}).  
By using this  $\sigma'$, and the above orientation, following the methodology of  Lemma \ref{Key}, we can first color all the edges incident with $ V(G) \setminus V(H)$, and then the edges in  $G'[V(H)]$, by processing the vertices one by one following   linear ordering $\sigma'$.     
This coloring process can be used to prove that  $\chi'_{k}(G) \leq 171k - 87$.   
The  proof for this comes from tweaking the proofs of Lemma \ref{Key}.   
As the authors in \cite{botler} have mentioned, we think  we would be far from  the truth still (refer Conjecture \ref{conj}), we skip the details of this small  improvement here for the readability of this paper.

%\acknowledgements
%\label{sec:ack}
%At the end of the manuscript, right before the bibliography you might
%want to place an acknowledgment. This can be easily done by using the 
%command \verb!\acknowledgements! as you can see here.

\nocite{*}
\bibliographystyle{abbrvnat}
% use the following instead if you encounter problems 
%\bibliographystyle{alpha}
\bibliography{mkkk-2024-dmtcs}

\begin{thebibliography}{7}
\providecommand{\natexlab}[1]{#1}
\providecommand{\url}[1]{\texttt{#1}}
\expandafter\ifx\csname urlstyle\endcsname\relax
  \providecommand{\doi}[1]{doi: #1}\else
  \providecommand{\doi}{doi: \begingroup \urlstyle{rm}\Url}\fi

\bibitem[Botler et~al.(2023)Botler, Colucci, and Kohayakawa]{botler}
F.~Botler, L.~Colucci, and Y.~Kohayakawa.
\newblock The mod {$k$} chromatic index of graphs is {$O(k)$}.
\newblock \emph{J. Graph Theory}, 102\penalty0 (1):\penalty0 197--200, 2023.
\newblock ISSN 0364-9024,1097-0118.
\newblock URL \url{https://doi.org/10.1002/jgt.22866}.

\bibitem[Hakimi(1965)]{SL-hak}
S.~L. Hakimi.
\newblock On the degrees of the vertices of a directed graph.
\newblock \emph{J. Franklin Inst.}, 279:\penalty0 290--308, 1965.
\newblock ISSN 0016-0032.
\newblock URL \url{https://doi.org/10.1016/0016-0032(65)90340-6}.

\bibitem[Mader(1972)]{Mader}
W.~Mader.
\newblock Existenz {$n$}-fach zusammenh\"angender {T}eilgraphen in {G}raphen
  gen\"ugend grosser {K}antendichte.
\newblock \emph{Abh. Math. Sem. Univ. Hamburg}, 37:\penalty0 86--97, 1972.
\newblock ISSN 0025-5858,1865-8784.
\newblock URL \url{https://doi.org/10.1007/BF02993903}.

\bibitem[Pyber(1992)]{pyber}
L.~Pyber.
\newblock Covering the edges of a graph by {$\ldots$}.
\newblock In \emph{Sets, graphs and numbers ({B}udapest, 1991)}, volume~60 of
  \emph{Colloq. Math. Soc. J\'anos Bolyai}, pages 583--610. North-Holland,
  Amsterdam, 1992.
\newblock ISBN 0-444-98681-2.

\bibitem[Scott(1997)]{scott}
A.~D. Scott.
\newblock On graph decompositions modulo {$k$}.
\newblock \emph{Discrete Math.}, 175\penalty0 (1-3):\penalty0 289--291, 1997.
\newblock ISSN 0012-365X,1872-681X.
\newblock URL \url{https://doi.org/10.1016/S0012-365X(96)00109-4}.

\bibitem[Thomassen(2014)]{Thomassen}
C.~Thomassen.
\newblock Graph factors modulo {$k$}.
\newblock \emph{J. Combin. Theory Ser. B}, 106:\penalty0 174--177, 2014.
\newblock ISSN 0095-8956,1096-0902.
\newblock URL \url{https://doi.org/10.1016/j.jctb.2014.01.002}.

\bibitem[Yang(2009)]{Yang}
D.~Yang.
\newblock Generalization of transitive fraternal augmentations for directed
  graphs and its applications.
\newblock \emph{Discrete Math.}, 309\penalty0 (13):\penalty0 4614--4623, 2009.
\newblock ISSN 0012-365X,1872-681X.
\newblock URL \url{https://doi.org/10.1016/j.disc.2009.02.028}.

\end{thebibliography}
\label{sec:biblio}

\end{document}